\documentclass[12pt]{article}
\usepackage{comment}
\usepackage{enumerate}
\usepackage{geometry} 
\usepackage{amsmath}
\usepackage{amsthm}
\usepackage{amssymb}
\usepackage{mathrsfs}
\usepackage{lipsum}
\usepackage{graphicx}
\graphicspath{{./figure/}}
\DeclareGraphicsExtensions{.pdf,.jpeg,.png,.jpg}
\usepackage{xcolor}
\usepackage{tikz}
\usepackage{tkz-euclide}
\newcommand{\ddt}[1]{\frac{\mathrm{d}#1}{\mathrm{d}t}}

\usepackage{hyperref}
\usepackage[capitalize]{cleveref} 
\usepackage{titletoc}
\usepackage{float}
\numberwithin{equation}{section}
\newcommand{\D}{\mathcal{D}}
\newtheorem{theorem}{Theorem}[section]
\newtheorem{lemma}[theorem]{Lemma}
\newtheorem{proposition}[theorem]{Proposition}
\newtheorem{corollary}[theorem]{Corollary}
\theoremstyle{definition}
\newtheorem{remark}[theorem]{Remark}
\theoremstyle{definition}
\newtheorem{definition}[theorem]{Definition}
\newtheoremstyle{withcitation}
{\topsep}{\topsep}{\itshape}{}{\bfseries}{.}{ }{\thmname{#1}\thmnumber{ #2}\thmnote{ (#3)}}
\theoremstyle{withcitation}
\newtheorem{lemmacite}[theorem]{Lemma}

\begin{document}
	\title{A prescribed curvature flow on hyperbolic surfaces with infinite topological type}
	\author{Xinrong Zhao, Puchun Zhou}
	\date{}
	\maketitle

        \begin{abstract}
In this paper, we investigate the prescribed total geodesic curvature problem for generalized circle packing metrics in hyperbolic background geometry on surfaces with infinite cellular decompositions. To address this problem, we introduce a prescribed curvature flow—a discrete analogue of the Ricci flow on noncompact surfaces—specifically adapted to the setting of infinite cellular decompositions. We establish the well-posedness of the flow and prove two convergence results under certain conditions. Our approach resolves the prescribed total geodesic curvature problem for a broad class of surfaces with infinite cellular decompositions, yielding, in certain cases, smooth hyperbolic surfaces of infinite topological type with geodesic boundaries or cusps. Moreover, the proposed flow provides a method for constructing hyperbolic metrics from appropriate initial data.
        \end{abstract}
        
	\section{Introduction}
Circle packing, as a discrete analog of conformal structure, was rediscovered by Thurston in his well-known note \cite{thurston1976geometry} for studying three-dimensional geometry and topology. Thurston conjectured that circle packings on the hexagonal triangulation can be used to approximate the Riemann maps, and the conjecture was proved by Rodin and Sullivan in \cite{Rodin_Sullivan}. The first existence result of circle packing on the plane was proved by Koebe \cite{koebe1936origin}, and was extended to all compact surfaces, which was now called the Koebe-Andreev-Thurston theorem (KAT theorem for short). After that, Colin de Verdi\`ere proposed an alternative proof of the existence and uniqueness of circle packings with a variational method, see \cite{MR1106755}. For other related variational methods, see the works of Br\"agger \cite{MR1189006}, Rivin \cite{MR1283870} and Leibon \cite{MR1914573}. For a systematic treatment, one is referred to the work of Bobenko and Springborn \cite{MR2022715}, which presents a comprehensive analysis of the energy functions associated with circle patterns.

The Ricci flow introduced by Hamilton \cite{hamilton1982three} is a very important tool in differential geometry and geometric analysis, which has many important applications, see e.g.  \cite{perelman2002entropy,MR2415394,MR2449060}. Studying the Ricci flows on noncompact manifolds is also a very important topic, which simultaneously serves as a key motivation for the work in this paper. On noncompact manifolds, even establishing the well-posedness of the Ricci flow is highly non-trivial.
 The short-time existence of Ricci flows on noncompact manifolds was first proved by Shi in \cite{Shi_noncompact}, where the boundedness of the Riemannian curvature tensor is assumed. And the uniqueness result was proved by Chen and Zhu in \cite{MR2260930}. After that, many important results on Ricci flows on noncompact manifolds were also established, see  \cite{MR2520796,MR4015429,MR4494617,MR2832165,MR3352241,ToppingYin24}.

Inspired by the convergence of Ricci flows on compact surfaces \cite{ham,chow1}, Chow and Luo introduced the combinatorial Ricci flow \cite{2003Combinatorial} which gives a parabolic proof of the KAT theorem. Moreover, the flow approach provides an effective algorithm for finding desired circle packings on surfaces with genus greater than zero. After that, the combinatorial Ricci flows in two and three dimensions were also used for studying ideal circle patterns \cite{ge2021combinatorial} and compact 3-manifolds \cite{feng2022combinatorial2}. Aside from the combinatorial Ricci flows, other combinatorial curvature flows were also introduced, which yield fruitful results. The combinatorial Yamabe flows were introduced by Luo \cite{luo2004combinatorial} for studying surfaces with discrete conformal structures. And the combinatorial Calabi flows are introduced by Ge \cite{ge_phd} for studying surfaces with constant discrete Gaussian curvatures. 

Circle packings on infinite triangulations are the discrete analog of conformal structures on noncompact surfaces. For infinite circle packings on the plane, important results on the existence and uniqueness of univalent circle packings were proved by He and Schramm, see Schramm \cite{Schramm_rigidity}, He-Schramm \cite{He_schramm} and He \cite{HE}.  

For corresponding parabolic results on infinite circle packings, Ge, Hua and the second author \cite{GeHuaZhou} first introduced the combinatorial Ricci flows on infinite disk triangulations in both Euclidean and hyperbolic background geometries, and proved the well-posedness of the flows and the convergence results. See also \cite{YuHao} for the infinite flow in the spherical case.

Recently, Ba-Hu-Sun studied the problem of prescribed total geodesic curvature on generalized hyperbolic circle packings in \cite{BHS}, which provides a method for studying hyperbolic circle packing metrics with cusps and geodesic boundaries.
The study of total geodesic curvatures on circle patterns originates from the
pioneer work of Nie \cite{nie2023circle}, where an energy function of ideal circle patterns on the sphere was introduced. It was proved that the spherical ideal circle patterns can be characterized by the total geodesic curvatures on circles. After that, total geodesic curvatures of generalized circle packings in hyperbolic background geometry on triangulated surfaces were studied in \cite{BHS}. Later, generalized hyperbolic circle packings on finite polygonal cellular decompositions were defined in \cite{HQSZ}, and the problem of prescribed total geodesic curvature was settled. Moreover, a prescribed curvature flow was introduced in \cite{HQSZ} for finding finite generalized circle packing metrics with prescribed total geodesic curvature.

To study hyperbolic circle packing metrics of infinite topological type, it is natural to consider the prescribed total geodesic curvature problem on infinite polygonal cellular decompositions. Motivated by the work of \cite{GeHuaZhou}, we introduce the prescribed curvature flow \eqref{eq:PCF} on such infinite cellular decompositions.
A fundamental question is whether there exists a generalized circle packing metric on an infinite cellular decomposition of a hyperbolic surface realizing a given total geodesic curvature. Such surfaces may possess infinitely many cone singularities, geodesic boundaries, and cusps. In this paper, we investigate the prescribed curvature flow on infinite cellular decompositions and prove several existence results using the flow method.

Let $\D=(V,E,F)$ be an infinite cellular decomposition that supports a generalized circle packing $\mathcal{C}=\{C_i\}_{i\in V}$ in hyperbolic background geometry; see Section \ref{sec:2} for details. 

\begin{definition}[Infinite prescribed curvature flow]\label{def:ricci-flow}
		Let $k_i$ be the geodesic curvature on circle $C_i$ and $s_i=\ln k_i$. Let $T_i$ be the total geodesic curvature of the circle $C_i$. For a prescribed total geodesic curvature $\hat{T}\in \mathbb{R}_+^V$, the \textbf{infinite prescribed curvature flow} (PCF) with initial value $s^0\in \mathbb{R}^{V}$ is defined as
        \begin{equation}\label{eq:PCF}
		\left\{\begin{aligned} &\frac{{\rm d} s_i}{{\rm d} t}=-(T_i-\hat{T}_i),  \quad \forall i \in V,\\&s(0)=s^0 .
        \end{aligned}\right.
		\end{equation}

	\end{definition}
      We first consider the well-posedness of the flow.
	
	\begin{theorem}[Well-posedness of the PCFs]
		Let $\mathcal{D}= (V, E, F)$ be an infinite polygonal cellular decomposition of a noncompact surface $S$. For any initial value $s^0$, there exists a solution $s(t)$, $t\in[0,\infty)$, to the flow \eqref{eq:PCF}. Moreover, if the cellular decomposition has bounded face degree, then the solution to the flow with uniformly bounded total geodesic curvatures on $V \times [0, \infty)$ is unique.
	\end{theorem}
 Here the \textbf{degree of a face} means the number of edges of the face. And the \textbf{degree} $\deg(v)$ of a vertex $v\in V$ refers to the number of faces incident to $v$.
 By \cref{cor:boundness} the total geodesic curvature $T_v$ of the circle $C_v$ is less than $\pi\deg(v)$ for each vertex $v\in V$. As a corollary, the PCFs are always well-posed when the cellular decomposition has a uniform bound on vertex degrees. 
 \begin{corollary}
     Let $\mathcal{D}= (V, E, F)$ be an infinite polygonal cellular decomposition of a noncompact surface $S$ with bounded vertex and face degree. For any initial value $s^0$, there exists a unique solution $s(t)$, $t \in [0, \infty)$, to the flow \eqref{eq:PCF}. 
 \end{corollary}
 
This is different from the Ricci flows in the continuous case, where uniqueness results do not hold for general noncompact manifolds.

    Moreover, we establish two convergence results for the prescribed curvature flows (PCFs) and derive several existence results for the prescribed total geodesic curvature problem.
        \begin{theorem}[Convergence of the hyperbolic PCFs I]\label{thm:convergenceI}
        Let $\mathcal{D} = (V, E, F)$ be an infinite polygonal cellular decomposition of a noncompact surface $S$. Let $\hat{T} = (\hat{T}_i)_{i \in V} \in \mathbb{R}_{+}^{V}$ be the prescribed total geodesic curvatures, and let $s^0$ be an initial value such that $T_i(s^0) - \hat{T}_i \geq 0$ for all $i \in V$. Then there exists a solution $s(t)$ to the flow~\eqref{eq:PCF} with initial value $s^0$, which converges to a generalized circle packing metric with total geodesic curvatures $\hat{T}$.
        \end{theorem}

	    With the help of Theorem \ref{thm:convergenceI}, we have the following corollary for generalized circle packing metrics on infinite triangulations.
        \begin{corollary}\label{smooth result}
            Let $\mathcal{T}=(V,E,F)$ be an infinite triangulation of a noncompact surface $S$. Then for each prescribed total curvatures $\hat{T}=(\hat{T_v})_{v\in V}\in \mathbb{R}^V_+$ satisfying 
            \begin{align}
                \hat{T}_v\le \mathrm{deg}(v),~\forall v\in V,
            \end{align}
            there exists a generalized circle packing $\mathcal{C}$ of $\mathcal{T}$ with the prescribed total geodesic curvatures $\hat{T}$. Moreover, all the circles are horocycles or hypercycles, which means the corresponding surface is a smooth surface with infinitely many geodesic boundaries or cusps.
        \end{corollary}
    \begin{remark}
        Since the hyperbolic surfaces obtained in Corollary \ref{smooth result} is smooth, Theorem \ref{thm:convergenceI} provides a method to construct noncompact hyperbolic surfaces with infinite topological type. To the best of our knowledge, there is currently no other method to establish the existence of such surfaces that support infinite circle packings with prescribed total geodesic curvatures.
    \end{remark}
	\begin{theorem}[Convergence of the hyperbolic PCFs II]\label{thm:convergenceII}
		Let $\mathcal{D}= (V, E, F)$ be an infinite polygonal cellular decomposition of a noncompact surface $S$. Let $\{P_m\}_{m\in\mathbb{N}}$ be an enumeration of all faces in $F$ and $\{a_n\}_{n\geq 1}$ be a strictly increasing sequence of positive integers. Define $H_n=\cup_{m=0}^{a_n} V(P_m)$. For any finite nonempty subset $W \subset V$, we denote by $n_W$ the smallest integer such that $W \subset H_{n_W}$.
        
        Let $\hat{T}=(\hat{T}_i)_{i\in V}\in \mathbb{R}_{+}^{V}$ be a prescribed total geodesic curvature such that for every finite nonempty subset $W\subset V$, we have

	    \begin{equation}\label{eq:UpperBoundForLhat}
	        \sum_{w\in W} \hat{T}_w < \sum_{m=1}^{n_W} \pi \min\{N(P_m,W), N(P_m)-2\}. 
	    \end{equation}
         If $s^0$ is an initial value such that $T_i(s^0)-\hat{T}_i\leq 0$ for all $i\in V$, then there exists a solution $s(t)$ to the flow \eqref{eq:PCF} with the initial value $s^0$, which converges to a generalized circle packing metric with total geodesic curvatures $\hat{T}$.
        	\end{theorem}
Here, the notations $V(P),~N(P)$ and $N(P,W)$ in Theorem \ref{thm:convergenceII} are defined as follows.

        \begin{definition}
        Let $V(P)$ denote the vertex set of a face $P$. For a face $P \in F$, we denote by $N(P)$ the number of vertices in $V(P)$. Given a finite subset $W \subset V$ and a face $P\in F$, the number $N(P, W)$ is given by $$N(P, W) = \#\{ v \in W \mid v \in V(P) \}.$$
         \end{definition}

        In Section~\ref{sec:4}, we will show that for any $\hat{T} \in \mathbb{R}_+^V$, if the initial value $s^0$ is sufficiently small, then $T_i(s^0) - \hat{T}_i \leq 0$ for all $i \in V$. Therefore, from \cref{thm:convergenceII}, we can derive the following corollary.
        
         \begin{corollary}\label{cor:convergenceII}
            Let $\mathcal{D}= (V, E, F)$ be an infinite polygonal cellular decomposition of a noncompact surface $S$. Suppose $\hat{T}=(\hat{T}_i)_{i\in V}\in \mathbb{R}_{+}^{V}$ is a prescribed total geodesic curvature that satisfies \eqref{eq:UpperBoundForLhat} for every finite nonempty set $W\subset V$, then there exists a generalized circle packing metric with the prescribed total geodesic curvature $\hat{T}$.
         \end{corollary}

        To formulate further existence results in certain special cases, we introduce the notion of a simple decomposition.
        \begin{definition}
            Let $\mathcal{D}= (V, E, F)$ be an infinite polygonal cellular decomposition. Fix a vertex $x_0\in V$. Define $B_0=\{x_0\}$, $B_{n+1}=\bigcup_{P\in F:V(P)\cap B_n\neq\varnothing} V(P)$ for $n\geq 0$ and let $\partial B_n=\{x\in B_n: \exists y\notin B_n, y\sim x\}$. We say $\mathcal{D}$ is a simple decomposition if it satisfies the following conditions:
            \begin{enumerate}[(1)]
                \item $B_{n+1} = B_n \cup \partial B_{n+1}$ for each $n \geq 0$.
                \item  $\sup_{x \in V} \deg(x)$ and $\sup_{P \in F} N(P)$ are finite.
                \item $\mathcal{D}$ has polynomial growth, that is, $|B_n| \leq Cn^k$ for some constants $C$ and $k$.
            \end{enumerate}
        \end{definition}

        The properties defined above hold for a large class of infinite cellular decompositions of surfaces, such as the square grid decomposition of the plane $\mathbb{R}^2$, as well as the tessellations of regular triangles and regular hexagons.
        
	    \begin{theorem}\label{thm:SimpleDecomposition}
	        Let $\mathcal{D}= (V, E, F)$ be a simple infinite polygonal cellular decomposition of a noncompact surface $S$. Suppose $\hat{T}=(\hat{T}_i)_{i\in V}\in \mathbb{R}_{+}^{V}$ is a prescribed total geodesic curvature such that for some fixed constants $\epsilon,\delta>0$ and every finite nonempty vertex set $W\subset V$ the inequality
                \begin{equation}\label{eq:SimpleDecomposition}
                    \sum_{w\in W}\hat{T}_w<(1-\frac{\delta}{n_W^{(1-\epsilon)}})\sum_{P} \pi \min\{N(P,W), N(P)-2\}
                \end{equation}
	        holds, where $n_W$ is the minimal integer such that $W\subset B_{n_W}$. Then there exists a generalized circle packing metric with the prescribed total geodesic curvature $\hat{T}$.
	    \end{theorem}

        The inequality~\eqref{eq:SimpleDecomposition} in \cref{thm:SimpleDecomposition} can be regarded as an approximate analog of the existence criterion in the finite case, see~\cite[Theorem 1.9]{HQSZ}.
        
        The organization of the paper is as follows. In Section \ref{sec:2}, we recall some basic definitions of generalized circle packing metrics on polygonal cellular decompositions, which were first introduced in \cite{HQSZ}. 
        In Section \ref{sec:3}, we establish the long-time existence of solutions to the flow \eqref{eq:PCF} with arbitrary initial values and then prove the uniqueness of the solution with bounded total geodesic curvatures. In Section \ref{sec:4}, we prove two convergence results of the flow \eqref{eq:PCF} under certain conditions and introduce several natural conditions on the prescribed total geodesic curvatures to ensure the existence of corresponding generalized circle packing metrics.

	\section{Preliminaries}\label{sec:2}
	\subsection{Generalized circle packings on polygons}
	We consider generalized circle packings in hyperbolic geometry, where we use ``generalized circles'' instead of ordinary hyperbolic circles. In this paper, generalized circles include the following three types of ``circles'', which are curves with constant geodesic curvatures in the hyperbolic space $\mathbb{H}^2$, as shown in Figure \ref{fig:cycles}.
	\begin{figure}
                \centering
                \includegraphics[width=0.43\linewidth]{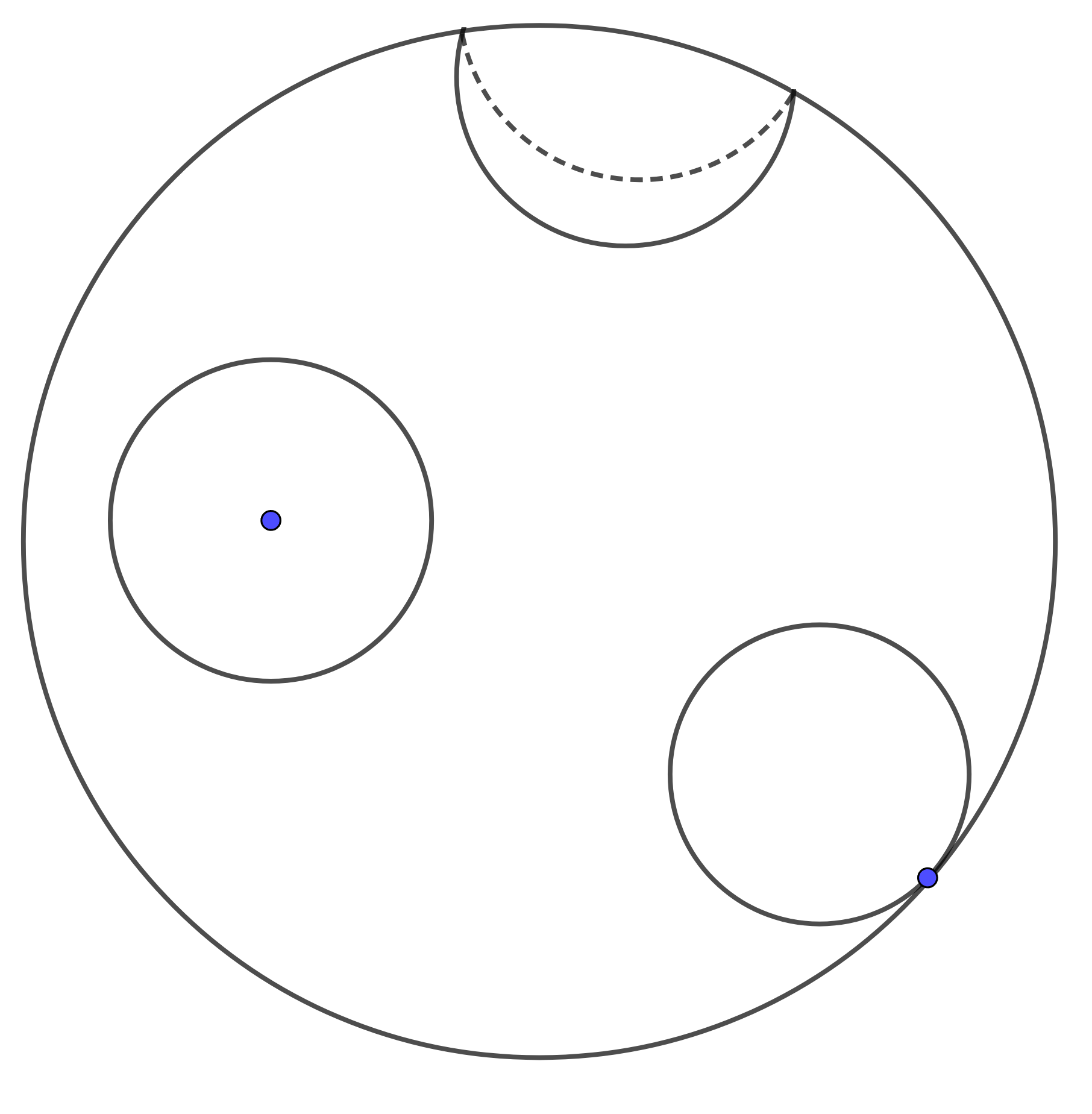}
                \caption{Circles, horocycles and hypercycles in the hyperbolic space $\mathbb{H}^2$.}
                \label{fig:cycles}
            \end{figure}
	\begin{enumerate}
	    \item A \textbf{hyperbolic circle} is a standard circle in the hyperbolic plane, which appears as an Euclidean circle entirely contained in the Poincar\'e disk model.
        \item A \textbf{hyperbolic horocycle} can be regarded as a circle with ``infinite radius''. Suppose $w \in \mathbb{H}^2$, and let $l$ be a geodesic ray starting at $w$. Let $\{p_i\}_{i=1}^\infty$ be a sequence of points on $l$ tending to infinity, and let $C_i$ be the hyperbolic circle centered at $p_i$ and passing through $w$. Then the limit of these hyperbolic circles $C_i$ is called a horocycle.
        
        In the Poincar\'e disk model, a hyperbolic horocycle is a Euclidean circle in the unit disk $\mathbb{D}^2$ which is tangent to its boundary $\partial\mathbb{D}^2$. In convention, a point on $\partial \mathbb{D}$ is called an ideal point. In this paper, we call this tangent point \textbf{the center of the horocycle}.  
		
		\item A \textbf{hypercycle} is a curve equidistant from a geodesic. Let $\gamma(t)$ be a complete geodesic in $\mathbb{H}^2$, that is, a geodesic defined for all real parameters $t\in\mathbb{R}$. The hypercycle with axis $\gamma$ and radius $r$ is a connected component of the set
		$$
		C(\gamma, r) = \{z \in \mathbb{H}^2 : d(z,\gamma) = r\},
		$$
		where $d(z, \gamma)$ is the distance from $z$ to $\gamma$. 
        
        In the Poincaré disk model, let $C$ be a Euclidean circle that intersects both the interior and the exterior of the unit circle. We denote by $C_1$ the portion of $C$ that lies within the unit disk $\mathbb{D}^2$. If $C_1$ is not a geodesic, then it is a hypercycle. \textbf{In this paper, we refer to the axis $\gamma$ of the hypercycle $C$ as its center, treating this curve as an abstract point when discussing generalized circle packings.} For an arc $\beta$ on a hypercycle $C(\gamma,r)$, we call the length of the projection of $\beta$ on the geodesic $\gamma$ the \textbf{central angle} of the arc $\beta$.
	\end{enumerate}

	We now define the generalized circle packing on polygons, which was first introduced in \cite{HQSZ}.
	
	\begin{definition}[Generalized circle packing on polygon \cite{HQSZ}]
		\label{def:gcp}
		Let $P$ be a topological polygon with $n \geq 3$ vertices $v_1, \dots, v_n$, arranged in clockwise order. Let $C_1, \dots, C_n$ be $n$ generalized circles in $\mathbb{H}^2$ corresponding to the vertices $v_1, \dots, v_n$. The configuration formed by $C_1, \dots, C_n$ is called a \textbf{generalized circle packing} of the polygon $P$ if the following conditions hold (see Figure \ref{fig:1}):
        \begin{enumerate}
          \item[\textnormal{(i)}] The generalized circles $C_i$ and $C_{i+1}$ are tangent to each other, where $1\leq i\leq n$ with the convention that $i + 1 = 1$ when $i = n$.
          \item[\textnormal{(ii)}] There exists a hyperbolic circle $C_P$ that is orthogonal to each $C_i$ at their tangency points, which is called the \textbf{dual circle} of the generalized circle packing.
        \end{enumerate}
	\end{definition}
	
	\begin{remark}
		Given a topological polygon $P$ with more than three vertices, if condition (ii) is not satisfied, then the configuration formed by the tangent circles is not unique up to isometry.
	\end{remark}
	
	The following lemma from \cite{HQSZ} shows the existence of generalized circle packings on polygons with prescribed geodesic curvatures.
	
	\begin{lemma}[\cite{HQSZ}]
		\label{lem:curvature}
		Given a topological polygon $P$ with vertices $v_1, \dots, v_n$, let $k_i \in \mathbb{R}^+$, $i = 1, \dots, n$. Then there exists a generalized circle packing $\{C_i\}_{i=1}^n$ of $P$ with the geodesic curvature $k_i$ on each $C_i$. Moreover, let $k_P$ be the geodesic curvature of the dual circle $C_P$, it turns out that $k_P$ is a $C^1$ continuous function of $(k_1, \dots, k_n)$.
	\end{lemma}

        \begin{figure}[htbp]
        \centering
          \includegraphics[width=0.6\textwidth]{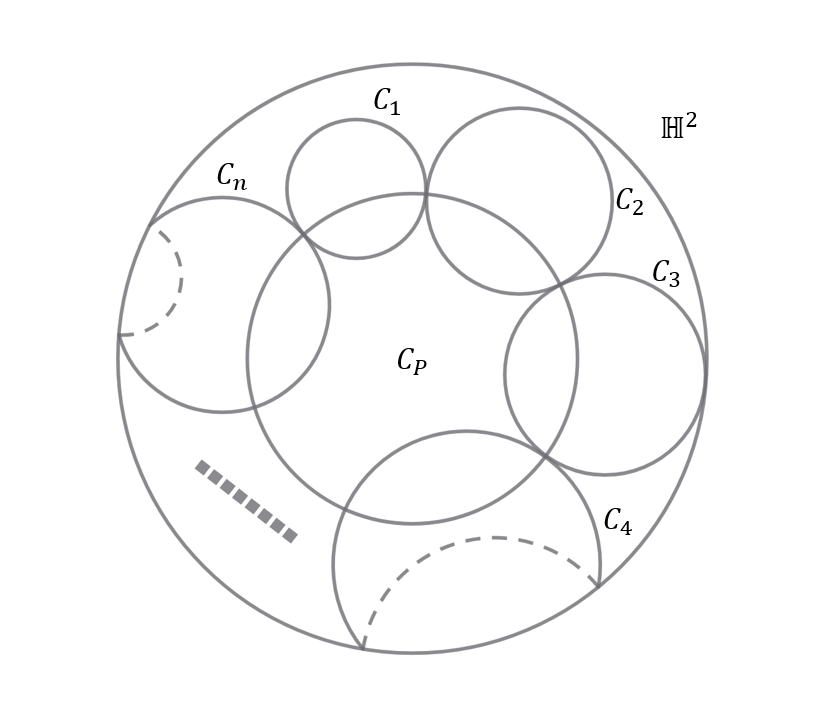} 
          \caption{A generalized circle packing.}
          \label{fig:1}
        \end{figure}
    
     We denote by $k_i$ the geodesic curvature of the circle $C_i$. We can classify the type of associated circle according to the value of geodesic curvature:
    \begin{enumerate}
        \item[(1)] $k_i>1$, the associated circle is a hyperbolic circle.
        \item[(2)] $k_i=1$, the associated circle is a horocycle.
        \item[(3)] $0<k_i<1$, the associated circle is a hypercycle.
    \end{enumerate}

	The associated radius of $C_i$ can be computed from its geodesic curvature as follows.
	$$
	r(k)=\left\{\begin{array}{ll}
		\operatorname{arctanh} k,  & 0<k<1, \\
		+\infty,  &k=1,  \\
		\operatorname{arccoth} k, & k>1.
	\end{array}\right.
	$$
	
	We now use $v_i$ to denote the center of $C_i$, and define the distance from $v_i$ to $v_j$ as $d(v_i, v_j) = r(k_i) + r(k_j)$. For each $1\leq i\leq n$, let $\gamma_{i,i+1}$ be the (possibly infinite) geodesic in $\mathbb{H}^2$ connecting $v_i$ and $v_{i+1}$ that realizes the distance from $v_i$ to $v_{i+1}$, where we again choose $i+1=1$ when $i=n$. Note that when $C_i$ is a hyperbolic circle and $C_{i+1}$ a hypercycle, $\gamma_{i,i+1}$ would be a geodesic starting from the point $v_i$ and perpendicular to the axis $v_{i+1}$. 
	
	\begin{definition}
		The hyperbolic polygon $\tilde{P}$ bounded by $\gamma_{i,i+1}$ and $v_i$ ($1\leq i\leq n$) is called the \textbf{generalized hyperbolic polygon} of a \textbf{generalized circle packing} with geodesic curvatures ${k} = (k_{v_1}, \dots, k_{v_n})$. 
	\end{definition}
For example, the region enclosed by the blue dashed curves together with the axis of $C_2$ in Figure \ref{fig:2} is a generalized hyperbolic pentagon with the generalized circle packing metric.

        \begin{figure}[htbp]
        \centering
          \includegraphics[width=0.6\textwidth]{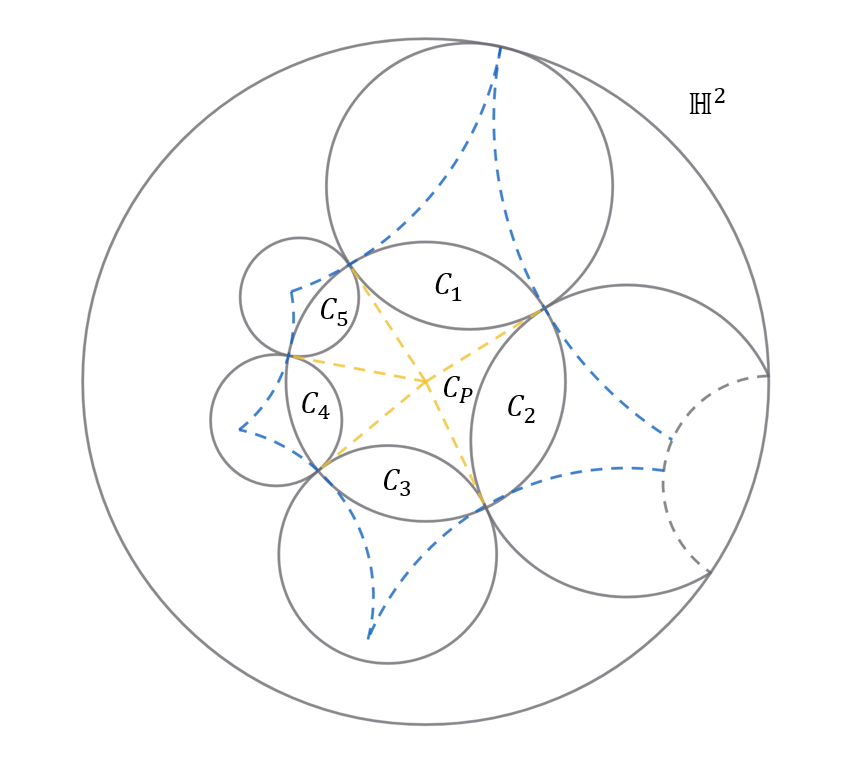} 
          \caption{A generalized pentagon of a generalized circle packing.}
          \label{fig:2}
        \end{figure}

		\subsection{Generalized Circle Packings on Surfaces}

		Let $\Sigma$ be the cellular decomposition of $S$. We denote by $V$, $E$ and $F$ the set of 0-cells, 1-cells and 2-cells of $\Sigma$, respectively.
		
		\begin{definition}[Infinite polygonal cellular decomposition]\label{def:polygonal}
			An infinite cellular decomposition $\Sigma$ is called an \textbf{infinite polygonal cellular decomposition}, if it satisfies:
			\begin{enumerate}
			    \item[(1)] Every 2-cell of $\Sigma$ is a polygon.
				\item[(2)] The number of 1-cells each 0-cell meets is finite and at least three.
				\item[(3)] The 1-skeleton is a simple graph without loops or double edges.
			\end{enumerate}
		\end{definition}

		Now we consider the generalized circle packing on surfaces with boundary in the hyperbolic background geometry. For simplicity, we make several conventions of the notations.
		\begin{enumerate}
			\item Use one index to denote a vertex ($i \in V$).
			\item Use two indices to denote an edge ($ij \in E$).
			\item Use $m$ indices to denote a face ($P = i_1i_2\cdots i_m \in F$ is the polygon on $S$ bounded by $i_1i_2, i_2i_3, \dots, i_mi_1$).
		\end{enumerate}
		
		\begin{definition}[Generalized hyperbolic circle packing]\label{def:circle-packing}
			Given the geodesic curvature ${k} = (k_i)_{i\in V} \in \mathbb{R}_+^{V}$ on $V$, we can construct a hyperbolic surface $\tilde{S} = \tilde{S}({k})$ as follows:
			\begin{itemize}
				\item[(1).] For each face $P=i_1i_2,...,i_m \in F$, we construct a generalized polygon $\tilde{P}$ of generalized circle packing with the geodesic curvature $(k_{i_1}, \dots, k_{i_m})$.
				\item[(2).] Gluing all the generalized polygons together by pairing of edges, with tangent coincidence, we obtain a hyperbolic surface $\tilde{S} = \tilde{S}({k})$, whose metric is called a \textbf{generalized circle packing metric}.
			\end{itemize}
		\end{definition}
		
		\begin{remark}\label{rem:partition}
			For any ${k} > 0$, there is a partition of $V$ given by
			\begin{align*}
				V_{hyp} &:= \{ i \in V \mid 0 < k_i < 1 \}, \\
				V_{horo} &:= \{ i \in V \mid k_i = 1 \}, \\
				V_{circle} &:= \{ i \in V \mid k_i > 1 \}.
			\end{align*}
			Topologically, the surface $\tilde{S} = \tilde{S}({k})$ can be obtained by removing disjoint open disks (or half of the open disks at the boundary) containing $V_{hyp}$ and points $V_{horo}$ from $S$. 
		\end{remark}


		Next, we give the definition of the total geodesic curvature.
		
		\begin{definition}[Total geodesic curvature]\label{def:total-curvature}
			The \textbf{total geodesic curvature} of an arc is defined as the integral of the geodesic curvature along the arc.
		\end{definition}
Now let $\{C_i\}_{i=1}^n$ and $C_P$ be a generalized circle packing of a polygon $P$ and its dual circle, as shown in Figure \ref{fig:1}. We denote by $k_i$ and $k_P$ the geodesic curvature of circles $C_i$ and $C_P$ respectively. Moreover, we denote by $D_P$ the interior of the dual circle $C_P$. Then we use $T_{i,P}$ to denote the total geodesic curvature of the arc $C_i\cap D_P.$ Then by the work in \cite{HQSZ}, the total geodesic curvature $T_{i,P}$ has the following expression.
        \begin{equation}\label{eq:ExpressionOfT}
            T_{i, P}=\left\{\begin{array}{ll}
        		\alpha_{i, P} \frac{k_i}{\sqrt{k_i^2-1}}, & k_i>1 \\
        		\frac{2}{k_P}, & k_i=1\\
        		\alpha_{i, P} \frac{k_i}{\sqrt{1-k_i^2}}, &0<k_i<1,
        	\end{array}\right. 
        \end{equation}
	where
	$$
	\alpha_{i,P}=\left\{\begin{array}{lr}
		2 \operatorname{arccot} \frac{k_P}{\sqrt{k_i^2-1}} & k_i>1,\\
		2 \operatorname{arccoth} \frac{k_P}{\sqrt{1-k_i^2}} & 0<k_i<1,
	\end{array}\right. \\
	$$
    is the central angle of the corresponding arc.
 \begin{definition}
     For a generalized circle packing, the total geodesic curvature $T_i$ at each vertex $i \in V$ is defined as 
     \[
     T_i=\sum_{P\in F:i\in V(P)}T_{i,P}.
     \]
 \end{definition}

\section{Well-posedness of PCFs}\label{sec:3}
	\subsection{Existence of the prescribed curvature flows}
	For the prescribed curvature flow on a finite polygonal cellular decomposition, the long-time existence of the solution follows from classical ODE theory and the boundedness of $T_i$ for each $i \in V$, as established by Hu-Qi-Sun-Zhou \cite{HQSZ}. In the spirit of Shi \cite{Shi_noncompact}, we employ approximation techniques to establish the long-time existence in the infinite case. The uniqueness of the solution will be addressed in the next subsection.

	\begin{theorem}\label{thm:existence}
		Let $\mathcal{D}= (V, E, F)$ be an infinite polygonal cellular decomposition of a surface $S$, then for any initial value $s^0 \in \mathbb{R}^V$, there exists a global solution $s \in C_t^2(V \times [0, \infty))$ to the flow \eqref{eq:PCF} with $s(t) \in \mathbb{R}^V$ for any $t \geq 0$.
	\end{theorem}
	
    Let $\mathcal{D}= (V, E, F)$ be an infinite polygonal cellular decomposition of $S$. Let $\{\mathcal{D}_i\}_{i\in\mathbb{N}}$ be an exhausting sequence of subcomplexes of $\mathcal{D}$, that is,
	$$
	\D_i \subset \D_{i+1}, \quad \cup_i \D_i = \D,
	$$
	and every compact subset of $S$ is contained in some $|\D_i|$ for sufficiently large $i$.
	
	We denote by $V_i$, $E_i$, and $F_i$ the sets of vertices, edges, and faces of $\D_i$ respectively. Let
        $$
        \partial_F V_i := \left\{ k \in V_i \,\middle|\, \exists\, j \notin V_i,\ P \in F \text{ such that } k, j \in V(P) \right\}
        $$
        be the set of boundary vertices of $V_i$, and define the set of interior vertices as
        $$
        \mathrm{int}(V_i) := V_i \setminus \partial_F V_i.
        $$
    
         Let $k=(k_i)_{i\in V}\in\mathbb{R}_{+}^V$ be a geodesic curvature vector. We consider the prescribed curvature flow on $T_i$ as follows.
	\begin{equation}\label{eq:FPCF}
		\left\{
		\begin{aligned}
			\frac{\mathrm{d} s_j^{[i]}(t)}{\mathrm{d} t} &= -(T_j-\hat{T}_{j}), && \forall j \in \mathrm{int}(V_i), \forall M > 0 \\
			s_j^{[i]}(t) &= s^0_j, && \forall (j,t) \in (V_i \times \{0\}) \cup (\partial_F V_i \times (0, \infty)),
		\end{aligned}
		\right.
	\end{equation}
	where the initial value is given by $s^0 = (\ln k_i)_{i \in V}$. Note that by the definition of $\mathrm{int}(V_i)$, the total geodesic curvature $T_j$ is well-defined for all $j\in\mathrm{int}(V_i)$.  Since $V_i$ is finite, the local existence of the above equations follows from Picard's theorem. The long-time existence of those flows \eqref{eq:FPCF} is similar to that in the finite case, see \cite{HQSZ}. We now prove the long-time existence of the infinite prescribed curvature flow.
	\begin{proof}[Proof of \cref{thm:existence} ]
		Fix a vertex $j \in V$, consider sufficiently large $i$ such that $j \in \mathrm{int}(V_i)$. Fix the time interval $[0,M]$ for $M > 0$. By \cref{cor:boundness}, we obtain the $C^1$ estimate of the flow
		$$
		\left|\frac{\mathrm{d} s_j^{[i]}(t)}{\mathrm{d} t}\right| \leq \deg(j)\cdot \pi+|\hat{T}_j|, \quad \forall t \in [0,M].
		$$
		
		Hence, from the flow equation \eqref{eq:FPCF}, it follows that
		$$
		|s_j^{[i]}(t)| \leq s^0_j + (\deg(j)\pi +|\hat{T}_j|)M.
		$$
		Therefore
		$$
		\sup_i \|s_j^{[i]}(t)\|_{C^1[0,M]} \leq C(s^0_j, \deg(j),|\hat{T}_j|, M).
		$$
		
		Moreover, differentiating the flow equation \eqref{eq:FPCF} yields
        $$
            \frac{{\rm d^2} s^{[i]}_{j}(t)}{{\rm d} t^2} =-\frac{{\rm d} T_j(s^{[i]}(t))}{{\rm d} t} = \sum_{k\in N(j)\cap \mathrm{int}(V_i)} \frac{\partial T_j}{\partial s_k} (T_k-\hat{T}_k) + \frac{\partial T_j}{\partial s_j} (T_j-\hat{T}_j)
        $$
		where 
        $$
        N(j):=\{k\in V: k\neq j, \exists P\in F\ \text{such that}\  k,j\in V(P)\}.
        $$
		
		Since $T_j$ is a $C^1$ function of $\{s_k\}_{k\in N(j)}$ and $s_j$, and the $C^1$ norm of $s_k^{[i]}(t)$ on $[0,M]$ is bounded for all neighbors $k \in N(j)$, we obtain the $C^2$ estimate
		$$
		\sup_i \|s_j^{[i]}(t)\|_{C^2[0,M]} \leq C_j
		$$
		for some constant $C_j$.
		
		Then, by the Arzel\`a-Ascoli theorem and a standard diagonal argument, there exists a subsequence $\{s^{[i_l]}(t)\}_{l=1}^\infty$ of $\{s^{[i]}(t)\}_{i=1}^\infty$ such that for each vertex $j \in V$, $s_j^{[i_l]}(t)$ converges in $C^1[0,M]$ to some $s_j^*(t)$ as $l \to \infty$, which satisfies the equation \eqref{eq:PCF}.
		
		Since this construction can be done for any $M > 0$, we obtain a global solution $s^*(t)$ to the flow \eqref{eq:PCF}. Moreover, since $T(s)$ is $C^1$ in $s$, the solution $s^*(t)$ is $C^2$ in $t$. This completes the proof.

	\end{proof}

        We now consider the maximum principle for the flow \eqref{eq:FPCF}. We first introduce some variational principles for the generalized circle packing metric on a polygon.

        \begin{lemmacite}[\cite{BHS}]\label{lem:sym}
		Let $\tilde{P}$ be the generalized polygon of generalized circle packing with respect to a topological polygon $P$, where the geodesic curvatures of vertices $v_1, \dots, v_n$ are $k_1, \dots, k_n \in \mathbb{R}^+$. Let $\{C_i\}_{i=1}^n$ be the generalized circle of the geodesic curvature $k_i$ and denote by $l_{i,P}$ the length of the sub-arc of $C_i \cap \tilde{P}$. Then the differential form
		$$ \omega_{\tilde{P}} = \sum_{i=1}^n l_{i,P} \, dk_i $$
		is a closed form.
	\end{lemmacite}
	
	\begin{remark}\label{rmk:symmetric}
		Recall that $T_{i,P}=k_i l_{i,P}$ is the total geodesic curvature of the sub-arc $C_i \cap \tilde{P}$,  $s_i=\ln k_i$, so $\omega_{\tilde{P}} = \sum_{i=1}^n l_{i,P} \, dk_i =\sum_{i=1}^n T_{i,P} \, ds_i$ is a closed form. In particular, $\frac{\partial T_{i,P}}{\partial s_j}=\frac{\partial T_{j,P}}{\partial s_i}$ for all vertices $i,j$ of $P$.
	\end{remark}
        
	\begin{lemmacite}[\cite{HQSZ}]\label{lem:sum}
		Let $C_1, \dots, C_n$ be a generalized circle packing of a topological polygon $P$ with geodesic curvatures $(k_1, \dots, k_n) \in \mathbb{R}_+^n$. Then
		$$
		\frac{\partial T_{i,P}}{\partial k_i} > 0, \quad \frac{\partial T_{i,P}}{\partial k_j} < 0 \quad \text{for all } i \neq j,
		$$
		and
		$$
		\frac{\partial\left(\sum_{j=1}^n T_{j,P}\right)}{\partial s_i} = k_i \frac{\partial\left(\sum_{j=1}^n T_{j,P}\right)}{\partial k_i} > 0, \quad \forall i = 1, \dots, n.
		$$
	\end{lemmacite}

        Then we can state the maximum principle for the flow \eqref{eq:FPCF} as follows.
	\begin{lemma}[Maximum principle I]\label{lem:MPI}
		Let $s^{[i]}(t)$ be a solution to the equation \eqref{eq:FPCF}. Let $T_j(s^{[i]}(t))$ be the total geodesic curvature at vertex $j \in \mathrm{int}(V_i)$. Set 
		$$
		m(t) = \min_{j\in\mathrm{int}(V_i)} \left(T_j(s^{[i]}(t))-\hat{T_j}\right)
		,M(t) = \max_{j\in\mathrm{int}(V_i)} \left(T_j(s^{[i]}(t))-\hat{T_j}\right).
		$$
		Then $m^*(t) := \min(m(t), 0)$ is nondecreasing, and $M^*(t) := \max(M(t), 0)$ is nonincreasing in $t$.
	\end{lemma}
	
	\begin{proof}
		Fix $i\in \mathbb{N}$. Since $s^{[i]}(t)$ is a solution to the equation \eqref{eq:FPCF}, for each $j \in \mathrm{int}(V_i)$ we have curvature evolution equation
        \begin{equation}\label{eq:CurvatureEvolution}
             \frac{{\rm d} (T_j(s^{[i]}(t))-\hat{T}_j)}{{\rm d} t} = - \sum_{k\in N(j)\cap \mathrm{int}(V_i)} \frac{\partial T_j}{\partial s_k} \left(T_k(s^{[i]}(t))-\hat{T}_k\right) - \frac{\partial T_j}{\partial s_j} \left(T_j(s^{[i]}(t))-\hat{T}_j\right).
        \end{equation}
		where 
        $$
        N(j)=\{k\in V: k\neq j, \exists P\in F\ \text{such that}\  k,j\in V(P)\}.
        $$
		
		By \cref{lem:sym} and \cref{lem:sum}, we have 
		\begin{equation}\label{eq:sum_positive}
			\begin{aligned}
				\frac{\partial T_{j}}{\partial s_j}+\sum_{k\in N(j)\cap \mathrm{int}(V_i)}\frac{\partial T_{j}}{\partial s_k}&=\frac{\partial T_{j}}{\partial s_j}+ \sum_{k\in N(j)\cap \mathrm{int}(V_i)} \frac{\partial T_{k}}{\partial s_j}\\
				&\geq\frac{\partial\left(T_j+\sum_{k\in N(j)}T_{k}\right)}{\partial s_j}\\
				&\geq \sum_{P:j\in V(P)}\frac{\partial\left(\sum_{k\in V(P)}T_{k,P}\right)}{\partial s_j}>0.
			\end{aligned}
		\end{equation}
		
        Suppose, for contradiction, that $M^*(t)$ is not nonincreasing on $[0,+\infty)$, then there must exists $0<t_1<t_2$ such that $0\leq M(t_1)<M(t_2)$.
        Therefore, the vector $T(s^{[i]}(t_2))-\hat{T}$ attains a positive maximum $M^*(t_2)$ at some vertex $j$. By the curvature evolution equation \eqref{eq:CurvatureEvolution} and inequality \eqref{eq:sum_positive}, we have
        $$
		\frac{{\rm d} (T_j(s^{[i]}(t))-\hat{T}_j)}{{\rm d} t}< - \sum_{k\in N(j)\cap \mathrm{int}(V_i)} \frac{\partial T_j}{\partial s_k} \left[\left(T_k(s^{[i]}(t))-\hat{T}_k\right)-\left(T_j(s^{[i]}(t))-\hat{T}_j\right)) \right]\leq 0.
		$$
        This implies that $M^*(t)$ must decrease in a neighborhood of $t_2$, and hence $M^*(t)$ is nonincreasing on $[0, t_2]$. This contradicts the assumption that $M(t_1) < M(t_2)$. A similar argument shows that $m^*(t)$ is nondecreasing on $[0, \infty)$.
	\end{proof}
    \begin{remark}
        This kind of maximum principles for discrete curvature flows first appeared in the work of Chow and Luo \cite{2003Combinatorial}.
    \end{remark}
	The following corollary is a direct consequence of the above lemma.
	    
	    \begin{corollary}\label{cor:MPI}
		Let $s^{[i]}(t)$ be a solution to flow \eqref{eq:FPCF}. Then
		$$
		|T_j(s^{[i]}(t))-\hat{T}_j| \leq \max_{l\in\mathrm{int}(V_i)} |T_l(s^0)-\hat{T}_l|, \quad \forall j \in \mathrm{int}(V_i), \forall t \geq 0.
		$$
	\end{corollary}

	By \cref{thm:existence} and \cref{cor:MPI} , we have
	
	\begin{proposition}
		Let $s^0$ be an initial value such that the corresponding total geodesic curvatures are uniformly bounded. Then there exists a global solution $s(t)$ of the flow \eqref{eq:PCF} with the initial value $s^0$ such that the total geodesic curvature $T(s(t))$ is uniformly bounded on $V \times [0, \infty)$. Specifically,
        $$
		|T_j(s(t))-\hat{T}_j| \leq \sup_{l\in V} |T_l(s^0)-\hat{T}_l|, \quad \forall j \in V, \forall t \geq 0.
		$$
	\end{proposition}

\subsection{Uniqueness of the PCFs}
	We now establish the uniqueness of solutions to the infinite prescribed curvature flow, when the total geodesic curvatures are uniformly bounded along the flow.
	\begin{theorem}\label{thm:uniqueness}
		Let $\mathcal{D}= (V, E, F)$ be an infinite polygonal cellular decomposition of a noncompact surface $S$ with bounded face degree. Let $s(t)$ and $\hat{s}(t)$ be two solutions to the flow \eqref{eq:PCF} on $t \in [0, M]$ with the same initial value $s^0$. If $T(s(t))$ and $T(\hat{s}(t))$ are uniformly bounded on $V \times [0, M]$, then $s(t) \equiv \hat{s}(t)$.
	\end{theorem}

        Let $G= (V, E)$ be an infinite graph, and let $\{\omega(t)\}_{t\geq 0}$ be a one-parameter family of weights on $E$ for $t \in [0, M]$ with $M > 0$.
        For a function $u:V\times [0,M]\to\mathbb{R}$, we define the weighted Laplacian $\Delta_{\omega(t)} u : V \to \mathbb{R}$ at time $t$ by
        $$
            \Delta_{\omega(t)}u(i)=\sum_{j:j\sim i}\omega_{ij}(t)(u_j(t)-u_i(t)).
        $$
        We will use the following maximum principle from \cite{GeHuaZhou}.
	\begin{lemma}[Maximum principle II]
			Let $G= (V, E)$ be a connected infinite graph, and let $\{\omega(t)\}_{t\geq 0}$ be a one-parameter family of weights on $E$ for $t \in [0, M]$ with $M > 0$ such that
		\begin{equation}\label{eq:omegabound}
			\sum_{j:j\sim i} \omega_{ij}(t) \leq C, \quad \forall (i,t) \in V \times [0,M],
		\end{equation}
		where $C$ is a uniform constant. 
        Suppose a function $g: V \times [0,M] \to \mathbb{R}$ satisfies
		$$
		\frac{dg_i(t)}{dt} \leq \Delta_{\omega(t)}g(i) + h_i(t)g_i(t)\quad \forall\ (i,t)\in V \times [0,M]
		$$
        for some function $h: V \times [0,M] \to \mathbb{R}$.
        If $g$ is a bounded function in $V \times [0,M]$ with $g \leq 0$ at $t= 0$, and if $h \leq B$ for some constant $B$, then
		$$
		g_i(t) \leq 0, \quad \forall (i,t) \in V \times [0,M].
		$$
	\end{lemma}

	\begin{corollary}\label{cor:uniqueness}
		If $g(t)$ is a bounded solution to the equation
		$$
		\frac{dg_i(t)}{dt} = \Delta_{\omega(t)}g(i) + h_i(t)g_i(t)
		$$
		with $g(0) \equiv 0$ on $V$, where $\omega(t)$ satisfies \eqref{eq:omegabound}, and the function $h \leq B$ for some constant $B$, then for any $t \geq 0$, $g(t) \equiv 0$.
	\end{corollary}
	
	Next, we derive an equation for the difference between the two solutions of the flow, which is essential for proving the uniqueness.
	
	\begin{lemma}\label{lem:uniqueness}
		Let $G' = (V', E')$ be the graph with $V' = V$ and $j \sim i$ in $G$ if and only if $j \in N(i)$, where
        $$
        N(i)=\{j\in V: j\neq i, \exists P\in F\ \text{such that}\  i,j\in V(P)\}.
        $$
        Let $s$ and $\hat{s}$ be two solutions to the flow \eqref{eq:PCF} on a time interval $[0,M]$. Viewing $s$ and $\hat{s}$ as functions on $V'\times[0,M]$, we have
		$$
		\frac{d(s_i(t) - \hat{s}_i(t))}{dt} = \Delta_{\omega(t)}(s(t) - \hat{s}(t))(i)+ h_i(t)(s_i(t) - \hat{s}_i(t),
		$$
		where
        \begin{equation}\label{eq:omega_ij(t)}
            \omega_{ij}(t) := - \int_0^1 \frac{\partial T_i}{\partial s_j}(\tau s(t) + (1-\tau)\hat{s}(t))\,{\rm d} \tau > 0,
        \end{equation}
        is a one-parameter family of weights on $E'$, and
		$$
		h_i(t) = -\int_0^1 \left( \frac{\partial T_i}{\partial s_i}(\tau s(t) + (1-\tau)\hat{s}(t)) + \sum_{k:k\sim i} \frac{\partial T_i}{\partial s_k}(\tau s(t) + (1-\tau)\hat{s}(t)) \right) {\rm d}\tau < 0,
		$$
	\end{lemma}
	\begin{proof}
		Let $s(t)$ and $\hat{s}(t)$ be two solutions to the flow. Then, for each $i \in V$,
		$$
		\begin{aligned}
			&\frac{d\left(s_i(t)-\hat{s}_i(t)\right)}{d t}\\
			=&-\left(T_i(s(t))-T_i(\hat{s}(t))\right) \\
			=& -\sum_{j:j\sim i} \int_0^1\frac{\partial T_i}{\partial s_j}(\tau s(t)+(1-\tau) \hat{s}(t)) {\rm d} \tau\cdot\left[(s_j(t)-\hat{s}_j(t))-(s_i(t)-\hat{s}_i(t))\right]\\
			&-\int_0^1\left( \frac{\partial T_i}{\partial s_i}(\tau s(t) + (1-\tau)\hat{s}(t)) + \sum_{k:k\sim i} \frac{\partial T_i}{\partial s_k}(\tau s(t) + (1-\tau)\hat{s}(t)) \right) {\rm d} \tau \cdot (s_i(t)-\hat{s}_i(t))\\
			=&\sum_{j:j\sim i}\omega_{ij}(t)\left[(s_j(t)-\hat{s}_j(t))-(s_i(t)-\hat{s}_i(t))\right]+ h_i(t)(s_i(t)-\hat{s}_i(t)).
		\end{aligned}
		$$
		The fact that $\omega_{ij}(t)>0$ and $h_i(t)<0$ follows from \cref{rmk:symmetric} and \cref{lem:sum}, which completes the proof.
	\end{proof}
In proving \cref{thm:uniqueness}, we first prove an estimate of the summation $\sum_{j:j\sim i}\frac{\partial T_i}{\partial s_j}$, with respect to the graph $G'$. 
        \begin{lemma} \label{lem:BoundnessOfSum}
        Let $T_i$ and $s_j$ be defined as before, we have
        $$
		\left|\sum_{j:j\sim i}\frac{\partial T_i}{\partial s_j}\right|\leq CT_{i},\ \text{for all $j\sim i$ in $G'$}.
		$$
        \end{lemma}
\begin{proof}
    
        Observe that
		$$
		\sum_{j:j\sim i} \frac{\partial T_i}{\partial s_j}=\sum_{P:i\in V(P)} \sum_{j\in V(P),j\neq i} \frac{\partial T_{i,P}}{\partial s_j}.
		$$
		By \cref{lem:sum}, for all $j\sim i$
		$$
		\frac{\partial T_{i,P}}{\partial s_j}=\frac{\partial T_{j,P}}{\partial s_i}<0,
		$$
		Therefore, it suffices to prove that for each polygon $P$
		$$
            -\sum_{j\in V(P),j\neq i} \frac{\partial T_{i,P}}{\partial s_j}\leq C T_{i,P}
        $$
        for some uniform constant $C$.
        
		By \cref{rmk:symmetric} and the expression \eqref{eq:ExpressionOfT} of $T_{i,P}$ we have
		$$
		\begin{aligned}
			-\sum_{j\in V(P),j\neq i} \frac{\partial T_{i,P}}{\partial s_j}=&-\sum_{j\in V(P),j\neq i} \frac{\partial T_{j.P}}{\partial s_i}\\
			=&k_i\left(\frac{2(k_P^2-1)}{k_P(k_P^2+k_i^2-1)}- \frac{2k_i(k_P^2-1)}{(k_i^2 +k_P^2-1)^2} \cdot \frac{1}{ \sum_{j=1}^n \frac{k_Pk_j}{(k_P^2 + k_j^2 - 1)}} \right).\\
		\end{aligned}
		$$
        Since the second term is nonpositive, it suffices to prove that
		$$
		\frac{2k_i(k_P^2 - 1)}{k_P(k_P^2 + k_i^2 - 1)} \leq C T_{i,P}.
		$$
		
		We proceed by analyzing three cases based on the value of $k_i$.
		
		  (1). When $k_i>1$, we have $T_i=\frac{2k_i}{\sqrt{k_i^2-1}}\operatorname{arccot} \frac{k_P}{\sqrt{k_i^2-1}}$, then we aim to prove
    		$$
    		\frac{k_P^2-1}{k_P^2+k_i^2-1}\leq C\frac{k_P}{\sqrt{k_i^2-1}}\operatorname{arccot} \frac{k_P}{\sqrt{k_i^2-1}}.
    		$$
    		Let $u=\frac{k_P}{\sqrt{k_i^2-1}}$, then the right-hand side becomes
    		$$
    		\text{RHS} = C u \operatorname{arccot}u=C u \operatorname{arctan}\frac{1}{u},
    		$$ 
    		while the left-hand side satisfies
                $$
                \text{LHS} \leq \frac{k_P^2}{k_P^2+k_i^2-1}=\frac{u^2}{1+u^2}.
                $$
                Thus, the inequality follows since $\frac{u^2}{1 + u^2} \leq C u \operatorname{arctan} \frac{1}{u}$ for some constant $C$.

            (2). When $k_1=1$, $T_i=\frac{2}{k_P}$. It is clear that
            $$
                \frac{2(k_P^2-1)}{k_P(k_P^2+k_i^2-1)}\leq \frac{2}{k_P}.
            $$
            
            (3). When $0<k_i<1$, we have
                $$
                T_i = \operatorname{arccoth} \frac{k_P}{\sqrt{1 - k_i^2}},
                $$
                and we aim to show that
    		$$
    		\frac{k_P^2-1}{k_P^2+k_i^2-1}\leq C\frac{k_P}{\sqrt{1-k_i^2}} \operatorname{arccoth} \frac{k_P}{\sqrt{1-k_i^2}}.
    		$$
    		This follows from the fact that $\frac{k_P}{\sqrt{1 - k_i^2}} > 1$, $\frac{k_P^2 - 1}{k_P^2 + k_i^2 - 1} < 1$, and $u \operatorname{arccoth} u \geq C_0$ for some positive constant $C_0$ when $u > 1$.

            This completes the proof.
\end{proof}
Now, we are able to prove \cref{thm:uniqueness} 

	\begin{proof}[Proof of \cref{thm:uniqueness}]
    Let $\mathcal{K}=\sup_{(i,t)\in V\times[0,M]}\left\{T_i(s(t))+T_i(\hat{s}_i(t))\right\}$. Assume that the face degree of $\mathcal{D}$ is bounded by a number $A$.
    
     By the equation 
     $$
     \left\{\begin{aligned} &\ddt{(s_i(t)-\hat{s}_i(t))}=-(T(s_i(t))-T(\hat{s}_i(t)),  \quad \forall i \in V,\\&s(0)-\hat{s}(0)=0,
        \end{aligned}\right.
     $$
     we have 
    \begin{align}\label{estimate flow}
        |s_i(t)-\hat{s}_i(t)|\le \mathcal{K}M,~\forall i\in V.
    \end{align}
    By \cref{cor:uniqueness} and \cref{lem:uniqueness}, it suffices to prove that $\omega(t)$ defined in \eqref{eq:omega_ij(t)} satisfies the assumption \eqref{eq:omegabound}. Recall that we are working on the graph $G'=(V',E')$ which is defined in \cref{lem:uniqueness}. 
        By equation \eqref{eq:omega_ij(t)} and \cref{lem:BoundnessOfSum},
        \begin{equation}\label{estimate1}
            \sum_{j:j\sim i}\omega_{ij}(t) = - \int_0^1 \sum_{j:j\sim i}\frac{\partial T_i}{\partial s_j}(\tau s(t) + (1-\tau)\hat{s}(t))\,{\rm d} \tau\leq C\int_0^1 T_i(\tau s(t) + (1-\tau)\hat{s}(t)){\rm d} \tau.
        \end{equation}
        
        We claim that 
        $$
        T_{i}(\tau s(t) + (1-\tau)\hat{s}(t))\le C'((T_{i}(s(t))+T_{i}(\hat{s}(t)))
        $$
        for any $\tau\in [0,1]$, where $C'$ is a uniform constant.
        
        To prove the claim, we proceed as follows. Fix $i\in V$, without loss of generality, we assume $s_i(t)\ge \hat{s}_i(t)$.
        Consider a face $P=i_1i_2...i_n\in F$ with $i\in V(P)$. After relabeling the indices if necessary, we assume that $i_1=i$. By the monotonicity in \cref{lem:sum}, we have 
        \begin{align}
          &T_{i,P}(\tau s_{i_1}(t) + (1-\tau)\hat{s}_{i_1}(t),\tau s_{i_2}(t) + (1-\tau)\hat{s}_{i_2}(t)...,\tau s_{i_n}(t) + (1-\tau)\hat{s}_{i_n}(t))\nonumber\\\le &T_{i,P}(s_i(t),\tau s_{i_2}(t) + (1-\tau)\hat{s}_{i_2}(t)...,\tau s_{i_n}(t) + (1-\tau)\hat{s}_{i_n}(t))
        \end{align}
      Furthermore, by \cref{lem:BoundnessOfSum} and \cref{lem:sum}, we have 
      \[
     \left| \frac{\partial\ln T_{i,P}}{\partial s_j}\right|\le C,
      \]
        for some uniform constant $C$ whenever $j\sim i$ in $G'.$
        Combining this with \eqref{estimate flow}, we obtain 
        \begin{align*}
        &\ln T_{i,P}(s_i(t),\tau s_{i_2}(t) + (1-\tau)\hat{s}_{i_2}(t)...,\tau s_{i_n}(t) + (1-\tau)\hat{s}_{i_n}(t))\\\le& \ln T_{i,P}(s(t))+CA\mathcal{K}M.
        \end{align*}
        Exponentiating both sides yields
        \[T_{i,P}(\tau s(t) + (1-\tau)\hat{s}(t))\le \exp(CA\mathcal{K}M)T_{i,P}(s(t))\le C'(T_{i,P}(s(t))+T_{i,P}(\hat{s}(t))).\] 
        The claim follows after summing over all faces $P\in F$ such that $i\in V(P)$.
        
        Substituting the claim into \eqref{estimate1}, we conclude that
        $$
            \sum_{j:j\sim i}\omega_{ij}(t)\leq C\int_0^1 T_i(\tau s(t) + (1-\tau)\hat{s}(t)){\rm d} \tau\leq  C'((T_{i}(s(t))+T_{i}(\hat{s}(t)))\le C'\mathcal{K},
        $$
        which completes the proof. 
        
    
	\end{proof}

	\section{Convergence of the PCFs}\label{sec:4}
        We begin by presenting results on the convergence of geodesic curvatures for a single polygon with $n$ vertices. Let $c = (c_1, \dots, c_n) \in [0, +\infty]^n$ be a nonnegative (possibly infinite) vector. Consider the limiting behavior of the sequence $\{\tilde{P}_m\}_{m=1}^\infty$ as
	$$ 
        \lim_{m\to+\infty} k^m = c = (c_1,      \dots, c_n). 
        $$
	Then we have the following two lemmas.

	    \begin{lemmacite}[\cite{HQSZ}]\label{lem:limitZero}
	        If $c_j = 0$ for some $j$ $(1 \leq j \leq n)$, then 
	        $$ \lim_{m\to+\infty} T_{j,P_m} = 0. $$
	    \end{lemmacite}
        
    	 \begin{remark}\label{rmk:limitZero}
    	A closer examination of the proof in \cite{HQSZ} reveals a stronger conclusion: for each $i \in V$ and any $\epsilon > 0$, there exists a positive constant $\delta_i > 0$ such that $T_i < \epsilon$ whenever $k_i < \delta_i$.
    	 \end{remark}

	\begin{lemmacite}[\cite{HQSZ}]\label{lem:limitInfty}
		Let $I \subset \{1, \dots, n\}$ be the nonempty vertex set that the corresponding circle degenerates to a point, that is,
		$$ I = \{i : c_i = \infty\}. $$
		Then
		\begin{align}
			\lim_{m\to+\infty} \sum_{i\in I} T_{i,P_m} &= |I|\pi, \quad \text{if } |I| \leq n - 3,  \\
			\lim_{m\to+\infty} \sum_{i\in I} T_{i,P_m} &= (n-2)\pi, \quad \text{if } n-3 < |I| \leq n. 
		\end{align}
	\end{lemmacite}

	\begin{proposition}[\cite{HQSZ}]\label{cor:boundness}
		Let $\{C_1, \dots, C_n\}$ be a generalized circle packing on a topological polygon $P$ with $n$ vertices. Then $T_{i,P} < \pi$ for each $i$.
	\end{proposition}
    
        Now we prove the convergence of the flow \eqref{eq:PCF} under the condition $T_i(s(t)) - \hat{T}_i \geq 0$. This leads to an existence result for noncompact hyperbolic surfaces of infinite topological type.
        
	\begin{proof}[Proof of \cref{thm:convergenceI}]
		Let $s(t)$ be the solution obtained in \cref{thm:existence} as the limit of a subsequence of the sequence $\{s^{[j]}(t)\}_{j=1}^{\infty}$. By \cref{lem:MPI}, we have $T_i(s^{[j]}(t)) - \hat{T}_i \geq 0$ for all $j \in \mathbb{N}$ and $(i, t) \in {\rm int}(V_j) \times [0, +\infty)$. Taking the limit as $j \to \infty$, it follows that $T_i(s(t)) - \hat{T}_i \geq 0$ for all $(i, t) \in V \times [0, +\infty)$. Hence, each $s_i(t)$ is nonincreasing in $t$.
		
		Suppose we have $s_i(t) \to -\infty$ as $t \to \infty$ for some $i \in V$. Then $k_i \to 0$, and by \cref{rmk:limitZero}, $T_i \to 0$, which contradicts the condition $T_i(s(t)) - \hat{T}_i \geq 0$. Therefore, each $s_i(t)$ is bounded below and converges to some value $s_i(\infty)$, which implies that $\frac{ds_i}{dt} = -(T_i(s) - \hat{T}_i)$ tends to $0$ for all $i\in V$. It follows that $T_i(s(\infty)) = \hat{T}_i$ for all $i\in V$.
	\end{proof}
	\begin{proof}[Proof of Corollary \ref{smooth result}]
	    Let $\mathcal{T}=(V,E,F)$ be an infinite triangulation of a noncompact surface. We consider a special generalized circle packing $\mathcal{C}$ of $\mathcal{T}$ such that all generalized circles are horocycles. Then each triangular face $P\in F$ has a circle packing metric with three horocycles as shown in Figure \ref{fig:triangular}. A simple computation shows that for each vertex $i\in P$, we have $T_{i,P}=1.$ Therefore, for the circle packing $\mathcal{C}$, we have $s_i=\ln k_i=0,T_i=\mathrm{deg}(i).$
        Consider the solution $s(t)$ of flow \eqref{eq:PCF} obtained in \cref{thm:existence} as the limit of a subsequence of the sequence $\{s^{[j]}(t)\}_{j=1}^{\infty}$ with $s^0=0\in \mathbb{R}^V$. We have $\hat{T}_{i}\le\deg(i)=T_i(s^0)$ for all $i\in V$.
        By Theorem \ref{thm:convergenceI}, the flow $s(t)$ must converge. Moreover, by the maximum principle, we have $s_i(t)\le s_i(0)=0,~\forall i\in V$, which means all the circles are hypercycles or horocycles.
\begin{figure}
            \centering
            \includegraphics[width=0.43\linewidth]{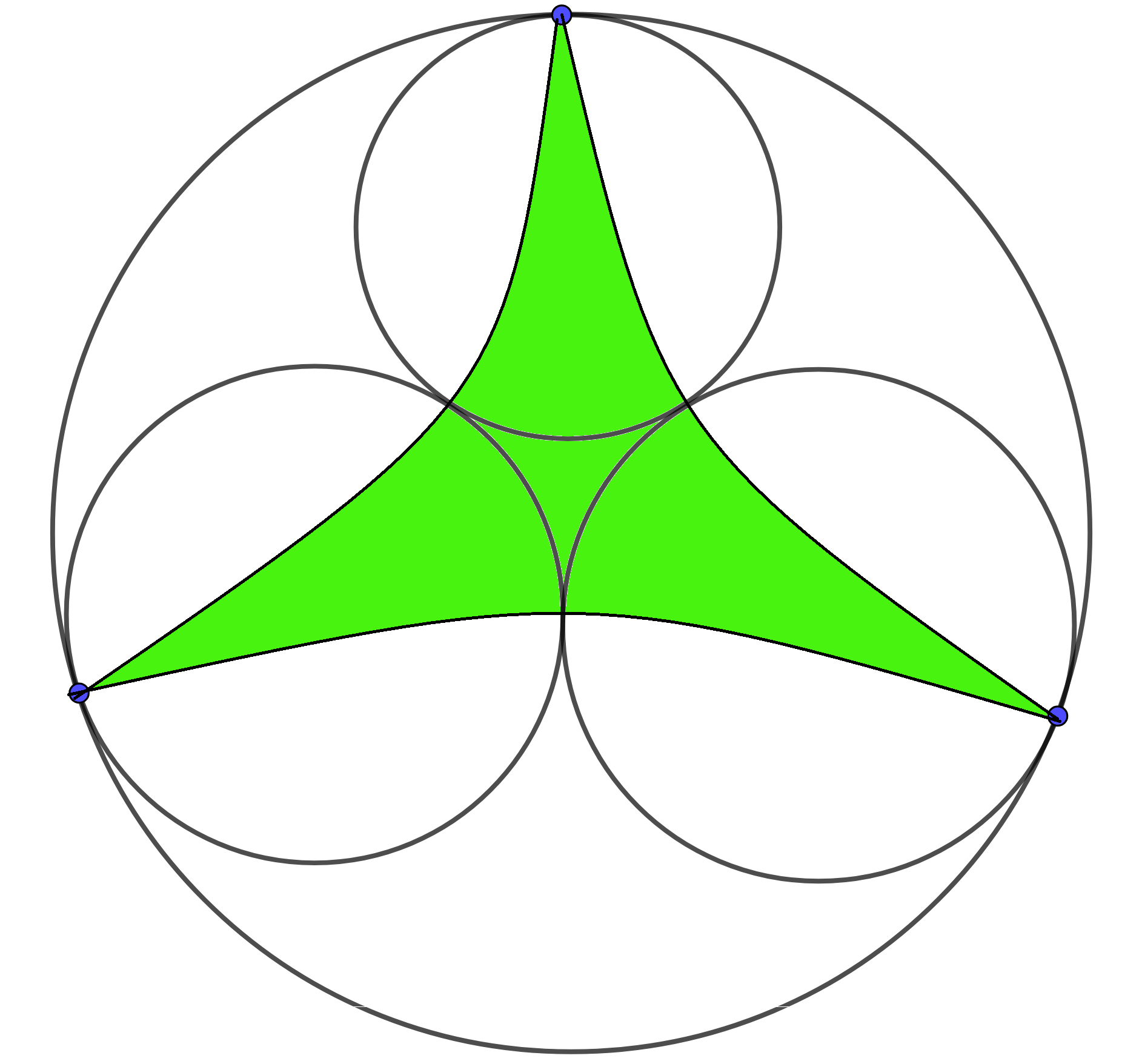}
            \caption{A generalized circle packing metric on a triangle that all circles are horocycles.}
            \label{fig:triangular}
        \end{figure}
                
	\end{proof}
	Next, we consider the case where $T_i(s(t)) - \hat{T}_i \leq 0$, with additional conditions.
	
	\begin{proof}[Proof of \cref{thm:convergenceII}]
		Again let $s(t)$ be the solution obtained in \cref{thm:existence} as the limit of a subsequence of $\{s^{[j]}(t)\}_{j=1}^{\infty}$. By \cref{lem:MPI}, we have $T_i(s^{[j]}(t)) - \hat{T}_i \leq 0$ for all $j$ and $(i,t) \in {\rm int}(V_j) \times [0,+\infty)$. Taking the limit as $j \to \infty$, we obtain $T_i(s(t)) - \hat{T}_i \leq 0$ for all $(i,t) \in V \times [0,+\infty)$. Hence each $s_i(t)$ is nondecreasing in $t$.
		
		Suppose $s_i(t) \to +\infty$ for some $i \in V$ as $t \to \infty$. Then there exists $N \in \mathbb{N}$ such that $i \in H_{N}$. Let 
        $$
        W=\{j\in H_{N}:\lim_{t\to\infty}s_j(t)=+\infty\}.
        $$
        Since $i\in W$, $W\neq\varnothing$. By \cref{lem:limitInfty}, we get
		$$
		\begin{aligned}
			\lim_{t \to \infty} \sum_{w \in W} T_w(t) 
			&\geq \lim_{t \to \infty} \sum_{m=1}^{a_{N}} \sum_{v \in W \cap V(P_m)} T_{v,P_m}(t) \\
			&= \sum_{m=1}^{a_N} \pi \min\{N(P_m,W), N(P_m) - 2\} > \sum_{w \in W} \hat{T}_w,
		\end{aligned}
		$$
		which contradicts the assumption $T_i(s(t)) - \hat{T}_i \leq 0$ for all $i\in V$. Therefore, each $s_i(t)$ is bounded above and converges to some value $s_i(\infty)$, which implies that $\frac{ds_i}{dt} = -(T_i(s) - \hat{T}_i)$ tends to $0$ for all $i\in V$. It follows that $T_i(s(\infty)) = \hat{T}_i$ for all $i\in V$.
	\end{proof}
	    
	    As a consequence of \cref{thm:convergenceII}, we obtain \cref{cor:convergenceII}.
	    
	 \begin{proof}[Proof of \cref{cor:convergenceII}]
        It suffices to prove the existence of an initial value \( s^0 \) such that \( T(s^0) \leq \hat{T} \). According to \cref{rmk:limitZero}, for each \( i \in V \), the total geodesic curvature \( T_i \) can be made arbitrarily small by choosing the corresponding geodesic curvature parameter \( k_i \) sufficiently small, independently of the values of \( k_j \) for \( j \neq i \). 
        
        Therefore, for each \( i \in V \), there exists a constant \( \delta_i > 0 \) such that \( T_i(k) \leq \hat{T}_i \) whenever \( k_i \leq \delta_i \). Defining \( s^0 = (\ln \delta_i)_{i \in V} \), we obtain \( T(s^0) \leq \hat{T} \) as desired. This completes the proof.
    \end{proof}       

        
            
	    
        Finally, we prove \cref{thm:SimpleDecomposition}.

	    \begin{proof}[Proof of \cref{thm:SimpleDecomposition}]
	        By \cref{rmk:limitZero}, there exists an initial value $s^0$ such that $T(s^0)\leq\hat{T}$. Using the same argument as in \cref{thm:convergenceII}, there exists a solution $s(t)$ to flow \eqref{eq:PCF} such that $s_i(t)$ is nondecreasing for every $i\in V$ and $T(s(t))\leq\hat{T}$ for each $t$. Define
	        $$
	        W=\{ i\in V: \lim_{t\to\infty} s_{i}(t)=\infty\}.
	        $$
	        It suffices to show that $W=\varnothing$.
            
	        Let $W_n=W\cap B_{n}$. Since
	        $$
	        \limsup_{t\to\infty}\sum_{w\in W_n} T_w\geq \sum_{P:V(P)\subset B_n} \pi\min\{N(P,W_n),N(P)-2 \},
	        $$
	        we have 
	        $$
	        (1-\frac{\delta}{n^{(1-\epsilon)}})\sum_{P} \pi \min\{N(P,W_n), N(P)-2\}\geq \sum_{P:V(P)\subset B_n} \pi\min\{N(P,W_n),N(P)-2 \}.
	        $$
	        Let $F_n=\{P\in F: V(P)\not\subset B_n, V(P)\cap B_{n}\neq \varnothing\}$. Then
	        $$
	        (1-\frac{\delta}{n^{(1-\epsilon)}})\sum_{P\in F_n} \min\{N(P,W_n), N(P)-2\}\geq \frac{\delta}{n^{(1-\epsilon)}}\sum_{P:V(P)\subset B_n} \min\{N(P,W_n),N(P)-2 \}.
	        $$
	        Since both $\sup_{P\in F} N(P)$ and $\sup_{x\in V} \deg(x)$ are finite, there exists a constant $C$ such that
	        $$
	        |W\cap\partial B_n|\geq C \frac{\delta}{n^{1-\epsilon}-\delta}|W\cap B_n|\geq C \frac{\delta}{n^{1-\epsilon}}|W\cap B_n|.
	        $$
	        Let $a_n=|W\cap B_n|$, then $a_n$ is nondecreasing and $a_n\geq \frac{C\delta}{n^{1-\epsilon}}\sum_{k=1}^n a_n$. Suppose $W\neq\varnothing$, then there exists a positive integer $N$ such that $a_N>0$. For each positive integer $k$,
	        $$
	        a_{2^{k+1}N}\geq \frac{2^kC\delta}{2^{(k+1)(1-\epsilon)}} a_{2^k N}.
	        $$
	        Iterating this, we obtain
	        $$
	            a_{2^{n}N}\geq (C\delta)^n {\sqrt{2}}^{n[(n-1)-(1-\epsilon)(n+1)]} a_N.
	        $$
	        This leads to a contradiction since $a_{2^{n}N}\leq |B_{2^{n}N}|\leq C'{2^{kn}}$ for some constant $C'$ and $k$.
	    \end{proof}
        
\textbf{Acknowledgements.} 
 We are grateful to Bobo Hua for his continuous support and valuable suggestions, which have greatly improved the writing of this paper. We also thank Huabin Ge for many insightful discussions on circle packing problems. In addition, we would like to thank Bohao Ji and Wenhuang Chen for their helpful discussions throughout the course of this research.
	\bibliographystyle{plain}
	\bibliography{main}

\noindent Xinrong Zhao, xrzhao24@m.fudan.edu.cn\\
\emph{School of Mathematical Sciences, Fudan University, Shanghai, 200433, P.R. China}\\[-8pt]

\noindent Puchun Zhou, pczhou22@m.fudan.edu.cn\\
\emph{School of Mathematical Sciences, Fudan University, Shanghai, 200433, P.R. China}
	  
\end{document}